\documentclass[a4paper,11pt]{article}
\usepackage{epsfig,amssymb,amsmath}
\usepackage{amsmath}
\usepackage{amssymb}
\usepackage{amsthm}
\usepackage{indentfirst}
\usepackage{xspace}
\usepackage{multirow}
\usepackage{verbatim}
\usepackage{authblk}
\usepackage{tikz-cd}
\usepackage{capt-of}
\usepackage{silence}
\usepackage{hyperref}
\usepackage{MnSymbol}

\usepackage{graphicx} 
\usepackage{epstopdf}
\usepackage{enumerate}
\usepackage{url}
\usepackage{bm}
\usepackage{epsfig}
\usepackage{color}
\usepackage[labelfont={rm}]{subcaption}

\usepackage{enumerate}
\usepackage[missing=]{gitinfo2}
\usepackage[mathscr]{euscript}

\numberwithin{equation}{section}

\usepackage[margin=1.1in]{geometry}

\newtheorem{theorem}{Theorem}[section]
\newtheorem{lemma}[theorem]{Lemma}
\newtheorem{proposition}[theorem]{Proposition}
\newtheorem{corollary}[theorem]{Corollary}

\newtheorem{question}[theorem]{Question}
\newtheorem{definition}[theorem]{Definition}

\newtheorem{theorem*}{Theorem}
\newtheorem{lemma*}{Lemma}
\newtheorem{proposition*}
{Proposition}
\newtheorem{corollary*}{Corollary}
\newtheorem{conjecture*}{Conjecture}

\newenvironment{remark}[1][Remark.]{\begin{trivlist}
\item[\hskip \labelsep {\bfseries #1}]}{\end{trivlist}}

\title{\textbf{Spectral networks, finite webs, and the $T_3$~cubic~differential}}

\author{Omar Kidwai\thanks{kidwai@math.cuhk.edu.hk}}
\affil{Department of Mathematics, The Chinese University of Hong Kong \newline Shatin, N.T., Hong Kong S.A.R., China}
\author{Guillaume Tahar\thanks{guillaume.tahar@gtiit.edu.cn}}
\affil{Department of Mathematics with Computer Science,\newline Guangdong Technion - Israel Institute of Technology (GTIIT)\newline
Shantou, Guangdong Province, China}
\date{}

\begin{document}

\maketitle

\vspace{-8mm}

\begin{abstract}
 We describe spectral networks for cubic differentials on $\mathbf{P}^1$ with exactly three poles all of order two, following observations of Hollands-Neitzke. To do so, we give a precise definition and establish some basic properties of a class of finite webs (corresponding to BPS states in physics) in the cubic differential setting, including their non-vanishing central charge. Turning to the example, we prove that at rational angles the network is totally degenerate; the number of walls and Stokes regions is explicitly computed and depends on an arithmetic criterion. At all other angles, the spectral network is dense in the surface and contains no finite webs, as expected. Physically, this example corresponds to a fundamental building block for a large class of four-dimensional $\mathcal{N}=2$ QFTs with gauge group $\mathrm{SU}(3)$.
\end{abstract}

\tableofcontents

\section{Introduction}

Spectral networks are graph-like objects on surfaces appearing in a a wide variety of applications. They are fundamental objects in asymptotic/WKB analysis \cite{Berk1982NewSL,aoki1991new,hondabook,hirose2025virtual,aoki1998exact, hollands2020exact}, quantum field theory \cite{gaiotto2010four,GAIOTTO2013239,Gaiotto:2012rg, Grassi:2021wpw}, and Donaldson-Thomas theory \cite{bridgeland2015quadratic,haiden,kidwai2024donaldson}, and are further connected to topics such as cluster algebras, \cite{iwaki2014exact, Neitzke17}, harmonic maps \cite{Neitzke17,dumasneitzke}, Higgs bundles \cite{GAIOTTO2013239, Neitzke:2014cja}, topological recursion \cite{IWAKI2022108191,iwaki2023topological}, ODE-IM correspondence \cite{Gaiotto:2014bza,dumasneitzke,ito2025odeim}, conformal field theory \cite{hao2025new,hollands2026liouville}, higher Teichm\"uller theory \cite{fock2006moduli,Gaiotto:2012db,kineider2024spectral,labourie2007flat}, and others.

Although a considerable amount is known about the simplest kinds of spectral networks, corresponding to those generated by trajectories of quadratic differentials, very little is known for the ``higher rank" case (but see \cite{casals2025spectral,kuwagaki2024generic,morrissey} for recent progress), especially explicitly. Some steps were recently taken by the authors in this direction in \cite{kidwai2025spectral}, including a convenient definition and some general properties for spectral networks coming from cubic differentials $\varphi$ on Riemann surfaces $X$ and the treatment of the first explicit examples, but much remains to be worked out. The fundamental difficulty in the higher rank case is that the network $\mathcal{W}$ is built up iteratively out of successive steps $\mathcal{W}^{(k)}$, $k\geq0$, with the novel phenomenon of intersection of trajectories dramatically increasing the complexity compared to the quadratic differential case.

The motivation of this paper is to work out completely the geometry of a fundamental example in this setting, whose corresponding BPS indices (fundamental invariants describing particles of the supersymmetric QFT, expected to coincide with Donaldson-Thomas invariants in algebraic geometry) were worked out to some order by Hollands and Neitzke \cite{Hollands:2016kgm}. Following their terminology, we refer to it as the \emph{$T_3$ cubic differential}. In that work, the analysis relied on spectral networks degenerating under specific conditions and having a certain structure, which we prove. In order to do so, we also provide some general definitions and results regarding \emph{finite webs} (corresponding to BPS states in physics or stable objects in algebraic geometry) in spectral networks coming from cubic differentials, which are expected to be of wider use in applications.

\paragraph{Context.}

In WKB analysis \cite{exactWKB,iwaki2014exact}, the spectral network is, roughly speaking, the locus where formal $\hbar$-series solutions of a certain Schr\"odinger-like equations cannot be ``resummed". On the complement $X\setminus\mathcal{W}(
\varphi)$, (the connected components of this complement are sometimes called \emph{Stokes regions}) it is expected that these solutions can be turned into genuine, analytic objects with appropriate asymptotic behaviour (i.e. are ``Borel summable"). In the context of the $T_3$ differential, Hollands-Neitzke performed detailed numerical studies of the WKB analysis of the corresponding ODE (``$T_3$ equation") \cite{hollands2020exact}. What we will prove therefore forms an essential step towards proving their conjectures in that work as well.

From another perspective, the work of Bridgeland-Smith \cite{bridgeland2015quadratic} read in conjunction with Gaiotto-Moore-Neitzke \cite{Gaiotto:2012rg} and Kontsevich-Soibelman \cite{kontsevich2008stability} suggests that the finite webs in this paper should correspond to stable objects in some 3-Calabi-Yau category. Following the physics intuition, the example we have treated here would also be an important building block for the understanding of such a correspondence. The analysis of the $T_3$ differential is thus a useful explicit step towards such a higher rank Bridgeland-Smith theory.

Spectral networks appear in quantum field theory as follows. They were originally introduced\footnote{Following their original observation in asymptotics much earlier in \cite{Berk1982NewSL}, see also \cite{hondabook}.} as fundamental objects in an algorithm for the determination of the so-called BPS spectrum (the list of stable supersymmetric ``BPS states") of a large class of four-dimensional $\mathcal{N}=2$ theories, known as \emph{class $\mathcal{S}$ theories} \cite{Gaiotto:2009we}, which take as initial data a Riemann surface $X$ and a Lie algebra $\mathfrak{g}$ (say, $sl_N$ --- $N=2$ then corresponds to quadratic differentials). The BPS states then arise as a parameter is varied as finite webs, which are a kind of degenerate portion of the network. Although this algorithm makes it straightforward (at least in principle) to obtain results, most studies depend on numerical techniques to guess the spectral networks corresponding to a given differential. 

\paragraph {The present work.} In this paper we study one of the two fundamental building blocks \cite{Gaiotto:2009we} corresponding to the case $N=3$, called the $T_3$ theory\footnote{Alternatively, and more historically, identified with the \emph{$E_6$ Minahan-Nemeschansky theory} \cite{MinahanNemeschansky1996E6}}. We prove that spectral networks degenerate precisely when expected by the analysis of Hollands-Neitzke.

More precisely, we determine the spectral networks $\mathcal{W}_\vartheta(\varphi)$  (here $\vartheta$ is a real parameter) for the $T_3$ differential on $X=\mathbf{P}^1$ given in a coordinate $x$ by

\[\varphi=\frac{\alpha dx^3}{x^2(x-1)^2}.\]

The main question one needs to answer to apply the result of Hollands-Neitzke is

\begin{question}
    When do the networks $\mathcal{W}_\vartheta(\varphi)$ degenerate (support so-called ``finite webs"), and what do they look like?
\end{question}

This question might be very hard to answer in general, and even more so when the cubic differential contains no poles of order greater than or equal to three; indeed, in this case, generically one expects trajectories to be dense in the surface. Despite this, due to the simple (the cover is a flat torus) and symmetric (defined by the Eisenstein lattice) geometry of this particular differential, we can answer the question exactly.

In order to prove this, we first give a careful definition of finite webs\footnote{Strictly speaking, we only treat finite webs \emph{supported on the spectral network}, which cannot veer off along other foliation trajectories, as in the closed loop or ``bubble" examples in \cite{Gaiotto:2012rg}.} and their canonical lifts in the general setting of cubic differentials. We prove the important fact that double walls (overlapping, oppositely oriented trajectories in the network) always come from such objects, and that their canonical lifts always have nonzero central charges (periods) of the relevant phase.

\begin{lemma}\label{lemintro:nonzero}
    Suppose $\varphi$ is a meromorphic cubic differential with at most simple zeroes.
    \begin{itemize}\item Every double wall in $\mathcal{W}_\vartheta(\varphi)$ is contained in the image of some finite web supported on the network.
    \item The central charge of any canonical lift of a finite web supported on $\mathcal{W}_\vartheta$ is nonzero of phase $\vartheta+\frac{\pi}{2}$ (up to rotations by $\frac{2\pi}{6}$).
    \end{itemize}
\end{lemma}

Results along these lines for quadratic differentials were used in e.g. \cite{IWAKI2022108191} to classify spectral networks in certain examples. It is hoped they should be of general use in classifying higher rank spectral networks too.

Together with an explicit analysis of trajectories on the torus, we may then answer Question 1.1 for the $T_3$ differential. Perhaps unsurprisingly, the answer depends on a notion of rationality (this was already noted in \cite{Hollands:2016kgm}). For any $\vartheta\in \mathbb{R}$, we say that $\vartheta$ is \emph{rational} if $\tan\vartheta \in \sqrt{3}\mathbb{Q} \cup\{\infty\}$ and \emph{irrational} otherwise. In the former case, we define $\tan \vartheta = \sqrt{3}\cdot\frac{q}{2p-q}$, with $p,q$ coprime and refer to the quadratic expression \[\Delta_\vartheta=p^2 - p q +q ^2\] as the \emph{Eisenstein determinant}. Our main theorem is then that

\begin{theorem}\label{thm:MAIN}
The spectral network $\mathcal{W}_\vartheta(\varphi)$ satisfies the following properties.
\par
If $\vartheta$ is rational with Eisenstein determinant $\Delta_{\vartheta}$, then $\mathcal{W}_\vartheta(\varphi)$ coincides with $\mathcal{W}^{(0)}_\vartheta(\varphi)$ and every trajectory of $\mathcal{W}_\vartheta(\varphi)$ is a double wall in the image of some finite web. Moreover, $\mathcal{W}_\vartheta(\varphi)$ decomposes $X$ into 
\begin{itemize}
    \item $\frac{2}{3}\Delta_{\vartheta}$ equilateral triangles of the same size if $\vartheta$ is ramified (i.e. $3 \mid \Delta_{\vartheta}$);
    \item $6\Delta_{\vartheta}$ equilateral triangles of the same size if $\vartheta$ is unramified (i.e. $3 \nmid \Delta_{\vartheta}$).
\end{itemize}

If $\vartheta$ is irrational, then $\mathcal{W}_\vartheta(\varphi)$ is dense in $X$ and does not support any double wall (nor finite web).
\end{theorem}

This confirms and sharpens what was observed by Hollands and Neitzke in \cite{Hollands:2016kgm}, and provides a rigorous geometric foundation for their computation of the BPS indices of the $T_3$ differential (as far as their finite computation went).

\paragraph{AI disclosure.}
AI assistance was used in drawing some figures in the paper, and in pointing out errors and inconsistencies which we fixed manually.

\paragraph{Acknowledgements.}

We thank Andy Neitzke and Lotte Hollands for feedback and for giving us permission to reproduce their plots in Appendix \ref{sec:conventions}.

O.K. is supported by the Improvement on Competitiveness in Hiring New Faculties Funding Scheme of The Chinese University of Hong Kong and the Hong Kong Research Grants Council Early Career Scheme award ``BPS Structures, Topological Recursion, and the Painlev\'{e} Equations" (Project Number 24309425). G.T. is supported by the French National Research Agency under the project TIGerS (ANR-24-CE40-3604).

\section{Spectral networks for cubic differentials}
Although we focus on the $T_3$ differential in the next section, all definitions and results in this section apply to arbitrary meromorphic cubic differentials, at least with the assumption of simple zeroes (this is a natural restriction appearing in most applications we have mentioned). In particular, the label-free definition of finite web, canonical lift, and properties of finite webs and the central charge for the setting of cubic differentials are new, general, and apply beyond just the $T_3$ case.

\subsection{Cubic differentials}\label{sub:differentials}

Given a compact Riemann surface $X$ with canonical bundle denoted by $\omega_X$, a (meromorphic) \emph{cubic differential} $\varphi$ is a meromorphic section of $\omega_X^{\otimes 3
}$. It is easy to see that the notions of \emph{zero}, \emph{pole}, and their \emph{orders} are well-defined (independent of a holomorphic coordinate).


\subsubsection*{Flat structure}

A meromorphic cubic differential on $X$ defines a geometric structure (an atlas of distinguished coordinate charts) in the following way. Denoting by $X^{\ast}$ the surface $X$ punctured at the zeros and poles of $\varphi$, any branch of $\int^{z} \varphi^{1/3}$ (with any fixed lower endpoint in $X^{\ast}$) defines a local biholomorphism from open subsets of $X^{\ast}$ to open subsets of $\mathbb{C}$.

The transition maps of this atlas of local coordinates are of the form
\begin{equation}
    z \mapsto \omega^r z + c
\end{equation} 
where $\omega$ is a primitive third root of unity, $r=0,1,2$, and $c$ is a constant. An atlas of this form is called a \emph{$\frac{1}{3}$-translation structure} (see \cite{BCGGM3,tahar} for references). To avoid such cumbersome terminology, in the rest of the paper, we will designate such a pair $(X,\varphi)$ as a \textit{flat surface}, with the understanding we really mean this particular class of flat surfaces. For sake of simplicity, we will sometimes use the single notation $\varphi$ to denote both the cubic differential $\varphi$ and the underlying complex structure $X$.
\par
In such a surface, the pullback of the standard metric of the Euclidean plane defines a flat metric $|\varphi^{1/3}|$ that is singular at the zeros, simple poles and double poles of $\varphi$ (poles of order at least three correspond to points at infinity for the flat metric, but are not present in this paper anyway). Furthermore, there is a globally defined notion of direction, or \emph{slope} (up to a rotation of angle $\frac{2\pi}{3}$), given by passing to any chart.

The local behaviour of $\varphi$ around any given point $p$ in $X$ is well-known (see e.g. \cite{BCGGM3} or \cite{tahar}). We refer to points which are either zeros or poles of $\varphi$ as \emph{singularities}, and all other points as \emph{regular points}. It is easy to see around any regular point that  $\varphi$ is locally biholomorphic to $(\mathbb{C},dz^3)$. Any singularity of order $k \in \mathbb{N}^{\ast} \cup \lbrace{ -1,-2, \rbrace}$ is called a \textit{conical singularity} of the flat metric. Indeed, some neighborhood of such a point is isometric to a neighborhood of the vertex of an infinite cone of angle $\frac{2\pi}{3}(k+3)$ (see Section 2.3 in \cite{BCGGM3} for details). Since these are the only singularities which appear in this paper, we omit the discussion of higher order poles.

\subsubsection*{Trajectories}\label{sub:Trajectories}

In $(X,\varphi)$, a \textit{trajectory} is an oriented arc parametrized by a real interval that does not contain any singularity (except possibly at the endpoints) and that is locally given by straight segments in the flat charts. We are typically interested in \textit{maximal} trajectories, whose endpoints (in the future) are either zeroes or poles of $\varphi$. We also use the following terminology:
\begin{itemize}
\item A \textit{critical trajectory} is a trajectory (of any slope) starting or ending at one of the conical singularities of $(X,\varphi)$.
\item A \textit{saddle connection} is a trajectory with both ends conical singularities (the same trajectory with an opposite slope is also a saddle connection).
\item A \textit{positive trajectory} (resp. \emph{negative trajectory}) is a trajectory whose oriented slope belongs to $\lbrace{ 0, \frac{2\pi}{3}, \frac{4\pi}{3} \rbrace}$ (resp. $\lbrace{ \frac{\pi}{3}, \pi, \frac{5\pi}{3} \rbrace}$) in any chart of the flat structure.
\end{itemize}

One of the basic behaviours of interest in applications is

\begin{definition}
    A (not-necessarily-maximal) trajectory in a spectral network is called a \emph{\textbf{double trajectory}} or \emph{\textbf{double wall}} if it is also a trajectory in the same spectral network when given the opposite orientation (and thus sign).
\end{definition}

\noindent More informally, a double wall is a pair of overlapping oppositely-oriented trajectories. Although double walls are not maximal, we always assume that they are ``maximally double", i.e. that they cannot be extended while remaining double walls. Unions of double walls (corresponding to so-called \emph{finite webs} in \cite{Gaiotto:2012rg}) appearing in a given spectral network then roughly correspond to BPS states in the physical picture. A saddle connection is an example of a double wall.

We refer to \cite{kidwai2025spectral} for further details on the behaviour of trajectories.
\subsubsection*{Canonical cover}
Given a compact Riemann surface with a cubic differential $(X,\varphi)$ which is assumed not to be globally a cube, there is always another compact connected Riemann surface $\Sigma$ with meromorphic one-form $\lambda$ and a holomorphic map $\pi:\Sigma\rightarrow X$ such that $\lambda^3=\pi^*\varphi$. Such a triple $(\Sigma,\lambda,\pi)$ of degree three is called the \emph{canonical cover} or \emph{spectral cover}. It is easy to check this property determines the triple uniquely up to biholomorphism.

\subsection{Spectral networks}

We recall here the definition of spectral network we will use, the notion of double wall, and simple but useful lemma on double walls.

\subsubsection*{Spectral networks} 
We recall the construction of spectral networks of cubic differentials in a simplified (but essentially equivalent) form to \cite{Gaiotto:2012rg}, introduced in \cite{kidwai2025spectral}.

\begin{definition}
For a meromorphic cubic differential $\varphi$, we define the \textbf{(WKB) spectral network} $\mathcal{W}(\varphi)$ of $(X,\varphi)$ by an iterative procedure:
\begin{itemize}
    \item The \textit{initial network} $\mathcal{W}^{(0)}=\mathcal{W}^{(0)}(\varphi)$ as the set of maximal trajectories formed by the positive and negative trajectories starting from conical singularities of $(X,\varphi)$.
    \item For any step $k$, the arc system $\mathcal{W}^{(k)}(\varphi)$ is formed by positive and negative trajectories (saddle connections account for both). We define $\mathcal{W}^{(k+1)}$ by adding to $\mathcal{W}^{(k)}$ a positive (resp. negative) maximal trajectory (the \emph{child}) starting at each point where two negative (resp. positive) trajectories of $\mathcal{W}^{(k)}$ (the \emph{parents}) intersect, see Figure~\ref{fig:joint}.
    \item The spectral network $\mathcal{W}(\varphi)$ of $\varphi$ is the union $\bigcup\limits_{k \geq 0} \mathcal{W}^{(k)}$, with each trajectory understood to be oriented.
\end{itemize}
\end{definition}

    \begin{figure}[h]
        \centering
        \begin{subfigure}{0.45\textwidth}
                \includegraphics[width=0.95\linewidth]{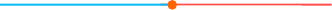}\vspace{2.15cm}
                \caption{Initial step near a double pole (trajectories are oriented outward). }
                \label{fig:zero}
        \end{subfigure}
        \hspace{0.25cm}
        \begin{subfigure}{0.45\textwidth}
        \hspace{0.75cm}\includegraphics[width=0.75\linewidth]{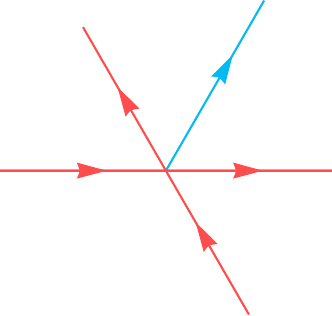}            \caption{Intersection with two positive trajectories giving birth to a negative trajectory.}\label{fig:joint}
        \end{subfigure}

        \caption{Construction of spectral networks from cubic differentials with second order poles.}
        \label{fig:construction}
    \end{figure}

We refer to the intersections between two trajectories of the same sign (i.e. those giving births) as \textbf{joints}.
 
    We use the notation $\mathcal{W}_\vartheta(\varphi):= \mathcal{W}(e^{-3i\vartheta}\varphi)$ to denote the spectral network of the rotated differential, which we call the spectral network of $\varphi$ ``of phase $\vartheta$".

We will later find it useful to keep track of the \emph{generation} of a wall $w$, the smallest $k$ such that $w\in \mathcal{W}^{(k)}$. The parents of a wall, and the parents of the parents, etc. of a wall are referred to as its \emph{ancestors}.

     As part of our strategy to describe the spectral network explicitly, we will refer to $\widetilde{\mathcal{W}_\vartheta}(\varphi)$ the ``spectral network on $\Sigma$"  by the exact same algorithm as above, except that the initial trajectories are taken to be the preimage in $\Sigma$ of the initial trajectories; the mod $\frac{2\pi}{3}$ slope then determines the birthing rule as before, where we consider trajetories on $\Sigma$ to be preimages of trajectories on $X$ (equivalently, geodesics in the flat geometry on $\Sigma$ determined by $\lambda$).

In applications, a fundamental role is played by double walls in the spectral network. The following easy lemma will be important for us:

\begin{lemma}\label{lem:doubleparents}
   The parents of a double wall are both double walls, at least between the birth of the parent and the birth of the child.
\end{lemma}
\begin{proof}
    If two walls give birth to a wall $w$ which is double, then $w$ with the opposite orientation will necessarily intersect to give a birth opposite and overlapping to each of the two parents, which are therefore double for at least as far back as the parents' births. See Figure \ref{fig:parentsdouble}.

  \begin{figure}[h]
        \centering
 
    \includegraphics[width=0.35\linewidth]{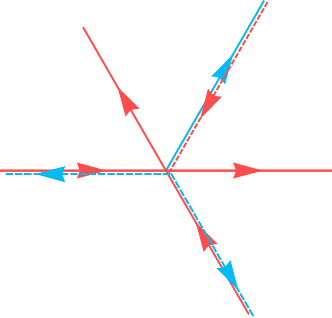}
    \caption{The double wall incoming (dotted red) to a joint produces double walls (dotted blue). Slight displacement for visibility.}
    \label{fig:parentsdouble}
\end{figure}
    
\end{proof}




\subsection{Finite webs, and canonical lifts}

We now describe the mathematical ingredients corresponding to so-called ``BPS states" in physics. We assume all zeroes of $\varphi$ are at most simple zeroes throughout, although this can probably be relaxed.

\subsubsection*{Lattice and central charge}

In applications, an important role is played by the homology lattice
\[
\Gamma = H_1(\widetilde{\Sigma},\mathbb{Z}).
\]

\noindent  where $\lambda$ is the meromorphic one-form on the spectral cover $\Sigma$, and $\widetilde{\Sigma}=\Sigma\setminus \mathrm{Pol}(\lambda)$ with $ \mathrm{Pol}(\lambda)$ the polar locus of $\lambda$. In physics, $\Gamma$ is the \emph{charge lattice} indexing the charges of particles; in Donaldson-Thomas theory, it is expected to be the Grothendieck group of the category in question.

A fundamental role is also played by the \emph{period} or \emph{central charge} homomorphism (viewing $\Gamma$ and $\mathbb{C}$ as abelian groups):

\begin{align*}
    Z:\,&\,\Gamma \longrightarrow \mathbb{C}\\
& \gamma \mapsto \int_{\gamma} \lambda.
\end{align*}

\noindent Note in the case of the $T_3$ differential, $\lambda$ is actually holomorphic, so $\widetilde{\Sigma}=\Sigma$ and $\Gamma=H_1(\Sigma,\mathbb{Z})$.

\subsubsection*{Finite webs}

One of the primary aims in the QFT context is to determine a set of homology classes corresponding to charges of physical particles. In our setting, these are captured by \emph{finite webs} as follows. Note that although we focus on the $T_3$ theory later on, the following definition makes sense for any meromorphic cubic differential.

\begin{definition}
Let $T$ be a (finite, connected) metric ribbon graph\footnote{That is, a graph with a cyclic orientation at each vertex and a length $l\in\mathbb{R}_{+}$ associated to each edge.} with trivalent ``internal" vertices and uni-valent ``external" vertices.
    A \textbf{finite web supported on the spectral network $\mathcal{W}_\vartheta(\varphi)$} is an orientation-preserving isometric immersion $\iota: T\rightarrow X$ such that each edge is sent to a double wall of $\mathcal{W}_\vartheta(\varphi)$, each internal vertex to a joint, and each external vertex to a conical singularity.
\end{definition}

An isometric immersion of a graph with two vertices connected by a single edge of length $l$ along a saddle connection of length $l$ is an example of a finite web. 
\begin{remark}
    Strictly speaking, in the physics context, there are cases where one is interested in finite webs which are \emph{not} supported on the spectral network. We ignore this technicality in this paper and refer loosely to finite webs supported on $\mathcal{W}_\vartheta$ simply as finite webs.
\end{remark}

\begin{remark}
    We also recall that there is a slight deficiency in our definition of spectral network \cite{kidwai2025spectral}: we do not incorporate so-called ``soliton data". This mean that we may have accidentally included some extra walls, so that we might in principle overcount the finite webs. In this paper we concern ourselves with showing where finite webs \emph{could} show up, keeping in mind that some may need to be discarded in applications.
\end{remark}

\begin{lemma}\label{lem:walltoweb}
    Every double wall in $\mathcal{W}_\vartheta(\varphi)$ lies in the image of some finite web supported on $\mathcal{W}_{\vartheta}(\varphi)$.
\end{lemma}

\begin{proof}
    We pick one of the two trajectories in the double wall and construct ``half" a finite web with one external vertex temporarily mapped to an arbitrary point $p$ along the interior of the double wall instead of a conical singularity (call such objects half-webs). 
    
    We now produce this half-web by induction on the generation $k$ of the wall in question.   
    
    If the wall is of generation $0$, i.e. it is a critical trajectory, we take $T$ to consist of two (external) vertices with an edge connecting them -- one mapped to the conical singularity, and one to $p$, with the edge mapped to the wall up to $p$, giving the requisite half-web.

    Suppose now that we have obtained a half-web from all double walls of generation at most $k$. Given a wall of generation $k+1$, we trace it backwards to the joint it was born from; by Lemma \ref{lem:doubleparents} both parents are double walls. Since the parents are at most generation $k$ by the inductive hypothesis we obtain two half-webs; gluing these two together at the vertex, together with the single edge from $p$ back to the vertex, we obtain a half-web.

Since the original wall was is double, we repeat the construction along the other trajectory in the opposite direction; then concatenating the two half-webs we have obtained at $p$, we obtain a genuine finite web.
    
\end{proof}

\subsubsection*{Canonical lifts}

Following Gaiotto-Moore-Neitzke \cite{Gaiotto:2012rg}, from a finite web we construct homology classes. As with our definition of spectral network, the construction is simplified thanks to the setting of cubic differentials, and not restricted any particular example.

\begin{definition}
   The \emph{\textbf{canonical lifts}} of a finite web supported on $\mathcal{W}_\vartheta(\varphi)$ are given by the following local lifting rule. 
   
   \begin{itemize}
        \item For every joint in the web, lift a portion of the trajectories along each ``bend" in the joint to successive sheets following the cyclic order as in Figure \ref{fig:jointlift}. Extend them all the way to the end of each incident double wall. Call each of these three 1-chains (six including orientation) a joint lift.

    \begin{figure}[h]
        \centering
     \includegraphics[width=0.85
    \linewidth]{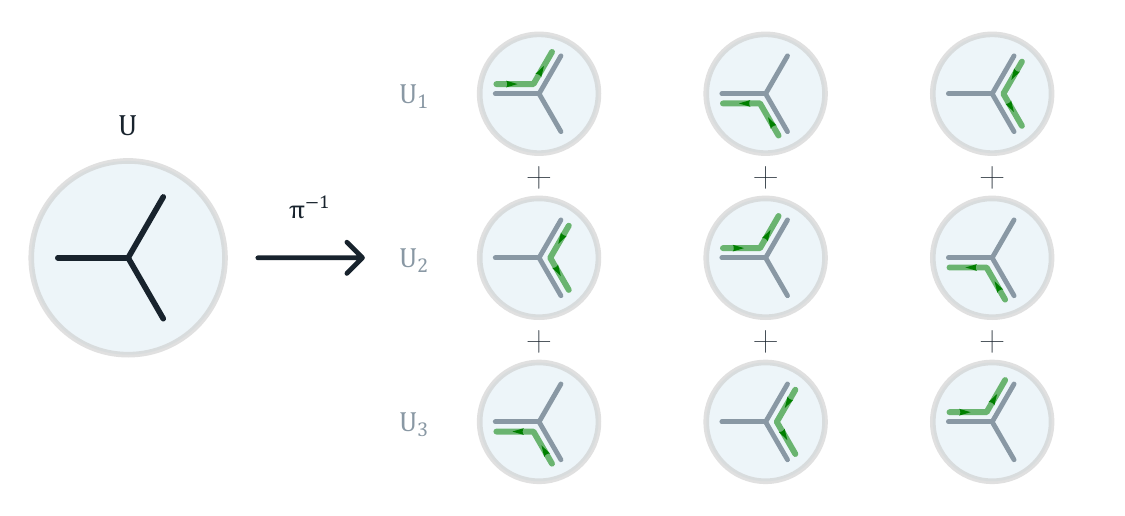}

        \caption{Joint lift (only one orientation is shown for each chain).}
        \label{fig:jointlift}
    \end{figure}

       \item For every double wall, lift the interior of its constituent trajectories to each of the three sheets. Call these wall lifts.

       \item For double walls with a conical singularity as an endpoint, we lift the wall to some sheet, follow the preimage of the wall back to the conical singularity,  traverse counterclockwise around it until the projection makes a full circle, and move back along the the preimage of the interior of the wall. Call these end lifts.

    \begin{figure}[h]
        \centering
        \begin{subfigure}{0.45\textwidth}
                \includegraphics[width=1.15\linewidth]{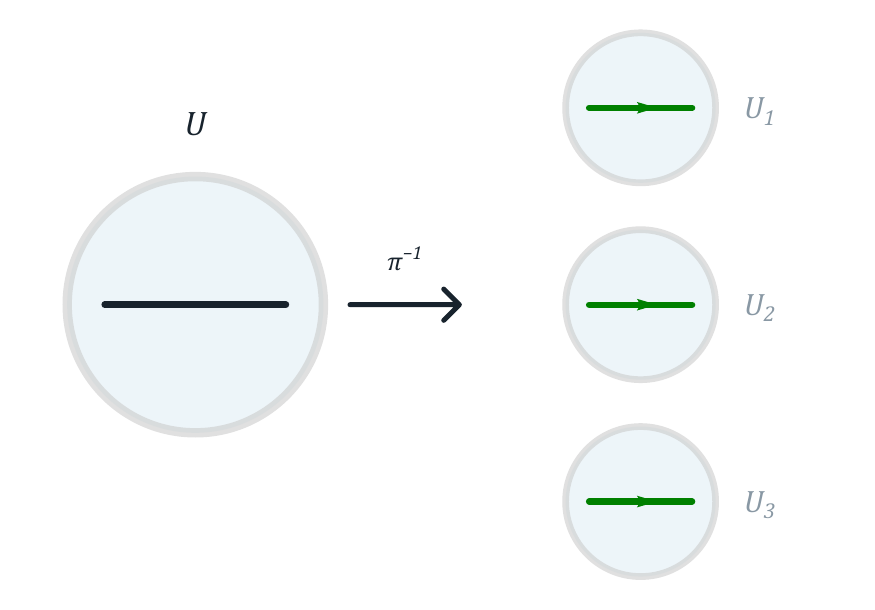}
                \caption{Wall lifts (only one orientation is shown).}
                \label{fig:walllift}
        \end{subfigure}
        \hspace{0.5cm}
        \begin{subfigure}{0.45\textwidth}
        \includegraphics[width=1.15\linewidth]{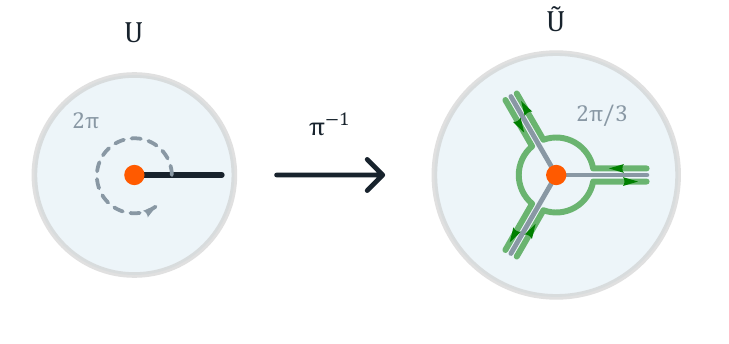}            \caption{End lifts around a ramification point of order three (only one orientation is shown).}\label{fig:endlift}
        \end{subfigure}

        \caption{Canonical lifting walls and near conical singularities.}
        \label{fig:wallendlift}
    \end{figure}

    \item Glue any wall, joint, and end lifts together, consistent with orientation, wherever they overlap.
   \end{itemize}

   By construction, we obtain at least one homology class in $\Gamma=H_1(\widetilde{\Sigma},\mathbb{Z})$ (in fact six, related by orientation reversal and the cyclic automorphism of $\Sigma$ permuting sheets).

\end{definition}

   

       

    



   Note that even if a vertex is not trivalent, we only consider trivalent embeddings --- for example in Figure \ref{fig:finiteweblift}, at the left vertex, there happen to be three trajectories which are not double (if they had been, the rule does not change, but we would have more embeddings).

    \begin{figure}[h]
        \centering
             \begin{subfigure}{0.37\textwidth}
        \includegraphics[width=\linewidth]{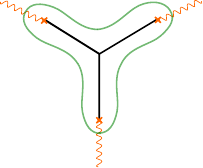}       
        \label{fig:canonicaltripod}
        \caption{Lifting a critical tripod}
        \end{subfigure}
                \hspace{1.5cm}
        \begin{subfigure}{0.33\textwidth}
                \includegraphics[width=\linewidth]{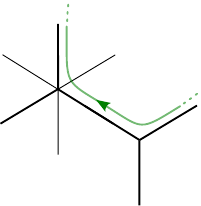}
                \vspace{0.5cm}
                \caption{A finite web might pass through a higher-valency vertex.}
                \label{fig:finiteweblift}
        \end{subfigure}

        \caption{Canonical lifting}
        \label{fig:canonicallifts}
    \end{figure}

A crucial point, itself of independent interest, is the following.

\begin{lemma}\label{lem:nonzero}
    Suppose $\varphi$ is a meromorphic cubic differential with at most simple zeroes. The central charge $Z(\gamma)$ of any canonical lift $\gamma\in\Gamma$ of a finite web supported on $\mathcal{W}_\vartheta(\varphi)$ is nonzero of phase $\vartheta+\frac{\pi}{2}$ (up to rotation by $\frac{2\pi}{6}$).
\end{lemma}

\begin{proof}
We prove it first locally around a joint; the general case follows. Assume $\vartheta=0$, which we may always do by first rotating $\varphi$.

Take a chart in the flat structure around a joint. Without loss of generality, we centre it at the origin in $\mathbb{C}$, the differential is given by $dz^{\otimes 3}$, and the spectral network in a neighbourhood of the joint looks as depicted in Figure \ref{fig:jointcharge}, with each of the three (double) trajectories a straight Euclidean line; call their lengths $l_a,l_b,l_c$. The central charge is given by integrating the three roots of $dz^{\otimes3}$ along the curves in the joint lift.

  \begin{figure}[h]
        \centering
 
    \includegraphics[width=0.45\linewidth]{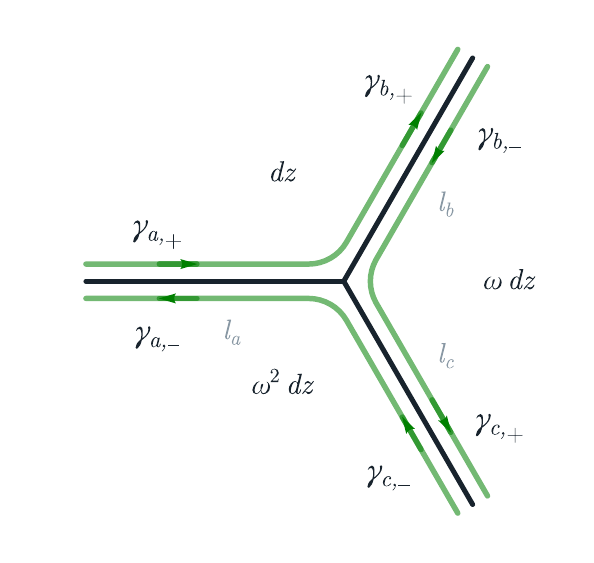}
    \caption{Central charge around a joint.}
    \label{fig:jointcharge}
\end{figure}

For $\gamma_{a,+}$, oriented positively along the real line, we have $\int_{\gamma_{a,+}}\lambda=l_a$. As we move to $\gamma_{b,+}$, we do not change sheets, so the integrand remains the same, but the path is rotated by $e^{\frac{\pi i}{3}}$, giving $\int_{\gamma_{b,+}}\lambda=e^{\frac{\pi i}{3}}l_b$. The contribution from $\gamma_{b,_-}$ has the same inclination but orientation reversed, and the lifted curve is on the cyclically-next sheet, so we must integrate $\omega dz$ i.e. $\int_{\gamma_{b,-}}=-\omega e^{\frac{\pi i}{3}} l_b$. Continuing until all six contributions have been obtained, we find

\[\int_{\gamma}\lambda=l_a(1-\omega^2)+ e^{\frac{\pi i}{3}}l_b(1-\omega)+e^{\frac{5\pi i}{3}}l_c(\omega-\omega^2) = (l_a+l_b+l_c)\sqrt{3} e^{\frac{i \pi}{6}}\]

\noindent which is therefore nonzero and of phase $\frac{\pi}{2}$ up to $\frac{2\pi}{6}$ rotation. A similar argument applies to the other lifts.

The general case follows; by the above, for any wall lift of a double wall $w$, the integral along the (compatibly oriented on the same sheet) wall-lift covering its parent double walls must have the same phase as the integral along $w$ itself; since finite webs are assumed to be connected and have finitely many edges, the resulting integral over the canonical lift of the whole web is a finite sum of pieces all of the same phase and nonzero length. The base case is merely a saddle connection, for which the result is immediate.

\end{proof}

\section{Analysis of the $T_3$ differential}
We now focus our attention to the main subject of the paper. We are interested in the spectral networks of cubic differentials on $\mathbf{P}^1$ with exactly three second order poles. Up to automorphisms of $\mathbf{P}^1$, such a differential can always be written in a coordinate $x$ as:

\[\varphi=\frac{\alpha dx^3}{x^2(x-1)^2}\]
Since $\alpha$ only overall changes the phases and rescales the metric, henceforth, we set $\alpha=1$. Since it is essentially unique, we refer to it (following terminology of \cite{Hollands:2016kgm}) as the \emph{$T_3$ cubic differential}.

Our analysis is possible thanks to the following crucial lemma.

\begin{lemma} The flat surface $\mathbf{P}^1$ endowed with the flat metric $|\varphi^{1/3}|$ is isometric to the conical flat surface obtained by gluing two equilateral triangles as shown. Its canonical triple cover (or spectral cover) $\Sigma$ ramified over the three double poles is the equilateral flat torus $\mathbb{C}/(\mathbb{Z}+\omega\mathbb{Z})$ with $\omega=e^{\frac{2i\pi}{3}}$. The three ramification points are respectively $A=0$, $B=\frac{2+\omega}{3}$ and $C=B+B=\frac{1+2\omega}{3}$.
\end{lemma}

\begin{proof}
The differential $\varphi$ has precisely three conical singularities, each of angle $\frac{2\pi}{3}$. Up to rescaling, it is the unique cubic differential on $\mathbf{P}^1$ with this property. On the other hand, by gluing two Euclidean triangles $\triangle_+$ and $\triangle_-$ as shown in Figure \ref{fig:base}, we immediately observe that each corner angle is $\frac{2\pi}{3}$. Thus the two surfaces are isometric up to scale.

    \begin{figure}[h]
        \centering
        \begin{subfigure}{0.35\textwidth}
                \includegraphics[width=\linewidth]{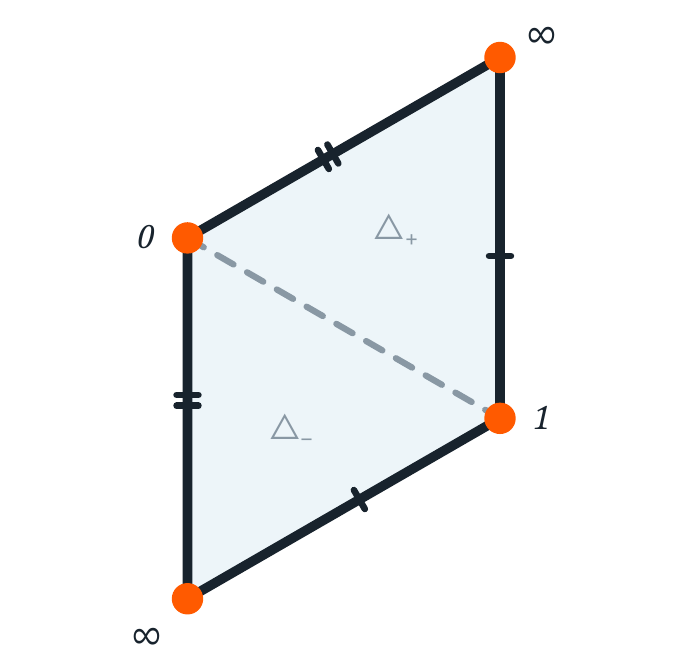}
                \caption{There are three singularities (conical points) on $X=\mathbf{P}^1$}
                \label{fig:base}
        \end{subfigure}
        \hspace{1.5cm}
        \begin{subfigure}{0.45\textwidth}
        \includegraphics[width=\linewidth]{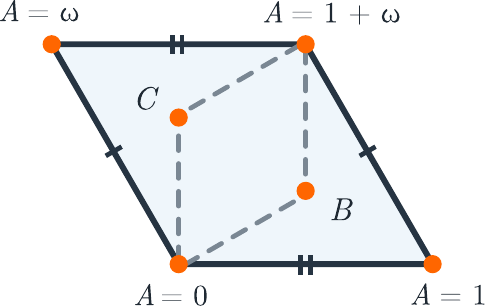}            \caption{The spectral cover $\Sigma$ is the equilateral torus }\label{fig:cover}
        \end{subfigure}

        \caption{The flat surface and its spectral cover.}
        \label{fig:basecover}
    \end{figure}
Taking $\Sigma$ to be an equilateral parallelogram with identified opposite sides, the map rotating the two triangles counterclockwise around the points $\frac{2+\omega}{3}$ and $\frac{1+2\omega}{3}$ gives a triple cover branched exactly at $A$, $B$ and $C$, see Figure~\ref{fig:cover}. Up to scaling and a rotation by $\frac{2\pi}{3}$, we obtain the flat torus  $\mathbb{C}/(\mathbb{Z}+\omega\mathbb{Z})$.
\end{proof}

Since $\Sigma$ is the quotient of $\mathbb{C}$ by the Eisenstein integers $\mathbb{Z}[\omega]$, we fix the $\mathbb{R}$-basis $(1,\omega)$ of $\mathbb{C}$, and denote by $R$ the operator corresponding to a rotation of angle $\frac{2\pi}{3}$. In this basis we have 
\[R(a,b)=(-b,a-b) \quad \quad R^{2}(a,b)=(b-a,-a).\] Note that $R$ preserves the Eisenstein determinant.

    Every trajectory of $\varphi$, in particular in $\mathcal{W}_\vartheta(\varphi)$ is the projection of some trajectory of $\lambda$ on $\Sigma$.

\begin{lemma}\label{lem:LiftProjection}
For any slope $\vartheta$, the preimage $\pi^{-1}({\mathcal{W}_\vartheta})$ coincides with the spectral network $\widetilde{\mathcal{W}}_\vartheta$ on $\Sigma$ with set of conical points $\lbrace{ A,B,C \rbrace}$.
\end{lemma}

Since $\Sigma$ is a torus, the networks depend on the rationality of the slope in the appropriate sense.

\begin{definition}\label{defn:RationalUnramified}
For any $\vartheta\in \mathbb{R}$, we say that $\vartheta$ is \textbf{\emph{rational}} if $\tan\vartheta \in \sqrt{3}\,\mathbb{Q}\cup\{\infty\}$ and \textbf{\emph{irrational}} otherwise.

When $\vartheta$ is rational, we write
$$
\tan \vartheta = \sqrt{3}\cdot\frac{q}{2p-q}
$$
with coprime integers $p,q$ (defined up to simultaneous change of sign). We then define the \textbf{\emph{Eisenstein determinant}}
$$
\Delta_\vartheta=p^2-pq+q^2.
$$
If $\vartheta$ is rational, we further say that it is \textbf{\emph{ramified}} if $3 \mid \Delta_\vartheta$ and \textbf{\emph{unramified}} otherwise.
\end{definition}

\begin{remark}
    Since our definition of spectral network involves a shift in the phase by $\frac{\pi}{2}$ compared to usual, the $p,q$ defined here are \emph{not} the same as in \cite{Hollands:2016kgm}. We clarify the relation in Appendix $\ref{sec:conventions}$.
\end{remark}

According to Definition~\ref{defn:RationalUnramified}, $\vartheta$ is a rational slope if it is directed by a vector $(p,q)$ where $p,q \in \mathbb{Z}$.

\begin{definition}\label{defn:DirectionalFoliation}
On the flat torus $\Sigma = \mathbb{C} / (\mathbb{Z}+\omega\mathbb{Z})$, for any oriented slope $\vartheta$, the (oriented) \textbf{directional foliation $\mathcal{F}_{\vartheta}$} (in the direction $\vartheta$) is the projection of the foliation of $\mathbb{C}$ by (oriented) parallel straight lines of slope $\vartheta$.
\end{definition}

The following lemma is easy to see.
\begin{lemma}\label{lem:Dichotomy}
On the flat torus $\Sigma$, for any oriented slope $\vartheta$, the directional foliation $\mathcal{F}_{\vartheta}$ satisfies one of the following statements:
\begin{itemize}
    \item $\vartheta$ is {rational} and every leaf is closed, representing the same homology class of $H_{1}(\Sigma,\mathbb{Z})$;
    \item $\vartheta$ is {irrational} and every leaf is dense in $\Sigma$.
\end{itemize}
\end{lemma}


\subsection{Lifted spectral networks: the case of rational slopes}\label{sub:rational}

We begin by describing the closed leaves of $\Sigma$ with rational slopes. 
We recall that for a rational slope $\vartheta$ directed by a primitive vector $v=(p,q)$, we have the Eisenstein determinant $\Delta_\vartheta=p^{2}-pq+q^{2}$. 
The description crucially depends on the value of $\Delta_\vartheta$ (note also that  $\Delta_\vartheta$ coincides with $\det(v\,|Rv)$).

\begin{lemma}\label{lem:Aligned}
In the canonical cover $\Sigma$, for a given rational slope $\vartheta$, one of the following statements holds: 
\begin{itemize}
    \item $\vartheta$ is ramified (i.e. $3 \mid\Delta_{\vartheta}$) and the points $A=0$, $B=\frac{2+\omega}{3}$ and $C=\frac{1+2\omega}{3}$ belong to the same closed leaf of slope $\vartheta$;
    \item $\vartheta$ is unramified (i.e. $3 \nmid\Delta_{\vartheta}$) and there is no closed leaf of slope $\vartheta$ containing more than one points among $A$, $B$ and $C$.
\end{itemize}

Moreover, $\Delta_{\vartheta}$ coincides with the number of intersections between a closed leaf of slope $\vartheta$ and a closed leaf of slope $R \vartheta$.
\end{lemma}

\begin{proof}
Since $C=B+B$ in $\mathbb{C}$, it follows that any closed leaf on $\Sigma$ containing both $A$ and $B$ will also contain $C$ and likewise for the other cases.
\par
For coprime integers $p,q$, let $(p,q)$ be a direction vector of the closed leaf $\ell$ of slope $\vartheta$ passing through $A=0$. The leaf $\ell$ contains $B=\left(\frac{2}{3},\frac{1}{3}\right)$\footnote{Here and throughout, vectors are written in the $\mathbb{R}$-basis $\langle 1$, $\omega\rangle$}if and only if there exist $t\in\mathbb{R}$ and $m,n\in\mathbb{Z}$ such that
$$
t \cdot (p,q)=\left(\frac{2}{3}+m,\frac{1}{3}+n\right).
$$
Eliminating $t$ gives the single Diophantine condition $3(qm-pn)+(2q-p)=0$.
\par
Since $\gcd(p,q)=1$, the map $(m,n)\mapsto qm-pn$ is surjective onto $\mathbb{Z}$, hence this condition is solvable if and only if $3 \mid (2q-p)$ which is equivalent to $3 \mid (p+q)$.
\par
As $\Delta_\vartheta \equiv (p+q)^2 \pmod 3$, this amounts to require that $\vartheta$ is ramified.

For the last statement, without loss of generality we shift the origin so that $\ell$ passes through 0. The condition for the leaf $\ell$ and $R\ell$ to intersect is for there to exist $s,t\in\mathbb{R}$, $m,n\in\mathbb{Z}$ such that
\[t\cdot(p,q)=s\cdot(-q,p-q)+(m,n)\]
which has solutions when $s=\frac{qm-pn}{p^2-p q + q^2}$. Discarding solutions which project to the same point (restricting to $0\leq s < 1)$, $s$ takes at most $p^2-pq+q^2$ values; by surjectivity again, we have exactly $p^2-pq+q^2$ values.
\end{proof}
In other words, when a rational $\vartheta$ is ramified, the network consists of ``one big leaf", so to speak, otherwise three closed saddles.

Now, we observe that whenever two critical trajectories starting at $A=0$ intersect, a third critical trajectory starting at $A$ passes through the same point.

\begin{lemma}\label{lem:TripleIntersection}
Let $v_1=(p,q)\in\mathbb Z^2$ be primitive and define
$v_2=R v_1$ and $v_3=R^2 v_1$. Let $\ell_i$ denote the closed leaf of the directional foliation of $\Sigma$ passing through $A=0$ in the direction of vector $v_{i}$.
\par
If a point $O \in \Sigma$ belongs to both $\ell_1$ and
$\ell_2$, then it also belongs to $\ell_3$. Moreover,  $\ell_{1} \cap \ell_{2} \cap \ell_{3}$ is formed by exactly $\Delta_\vartheta$ distinct points.
\end{lemma}

\begin{proof}
The intersection $\ell_1\cap \ell_2$ is the set of points $O$ of $\Sigma$ admitting a lift in $\mathbb R^2$ such that there exist $s,t\in\mathbb R$ and $(m,n)\in\mathbb Z^2$ such that $O = s v_1 = t v_2 + (m,n)$.
\par
Thus $(m,n)$ is of the form $s v_1 - t v_2=(sp + tq, sq+tq-tp)$. The following matrix equation is satisfied:
\[
\begin{pmatrix}
p&q\\
q&q-p
\end{pmatrix}
\begin{pmatrix}
s\\ t
\end{pmatrix}
=
\begin{pmatrix}
m\\ n
\end{pmatrix}.
\]
Since the determinant of $\begin{pmatrix}
p&q\\
q&q-p
\end{pmatrix}$
is $-\Delta_\vartheta$, we deduce that $s,t \in \frac{1}{\Delta_\vartheta} \mathbb{Z}$. In particular, $O \in \frac{v_{1}}{\Delta_\vartheta}\mathbb{Z}$.
\par
In the elliptic curve $\Sigma$ with the group law where $A=0$, $\frac{v_{1}}{\Delta_\vartheta}$ generates a cyclic subgroup $\Lambda$ of length $\Delta_\vartheta$. We proved above that any point of $\ell_1\cap \ell_2$ belongs to $\Lambda$. We already know that $\ell_1\cap \ell_2$ is formed by exactly $\Delta_\vartheta$ points (see the last claim of Lemma~\ref{lem:Aligned}) so $\ell_1\cap \ell_2$ coincides with $\Lambda$.
\par
Using the same argument, we show that $\ell_1\cap \ell_3$ coincides with $\Lambda$. It follows that $\ell_1\cap \ell_2$ coincides with $\ell_1\cap \ell_{2} \cap \ell_3$. The three pairwise intersections coincide.
\end{proof}

\subsubsection*{Ramified rational slopes}

When $\vartheta$ is a ramified rational slope, it follows from Lemma~\ref{lem:Aligned} that all the critical trajectories of $\widetilde{\mathcal{W}}^{(0)}_{\vartheta}$ reduce to the six trajectories emanating from $A$, since each of them necessarily contains both $B$ and $C$.
\par
At this point, we can give a complete description of the lifted spectral network on $\Sigma$ for ramified rational slopes.

\begin{corollary}\label{cor:RamifiedSlope}
Consider a ramified rational slope $\vartheta$ with Eisenstein determinant $\Delta_{\vartheta}$. Then, the lifted spectral network $\widetilde{\mathcal{W}_\vartheta}$ on $\Sigma$ satisfies the following conditions:
\begin{itemize}
    \item Every trajectory of $\widetilde{\mathcal{W}_\vartheta}$ is a double wall;
    \item $\widetilde{\mathcal{W}}_\vartheta$ coincides with $\widetilde{\mathcal{W}}^{(0)}_\vartheta$;
    \item $\widetilde{\mathcal{W}}_\vartheta$     decomposes $\Sigma$ into $2\Delta_\vartheta$ equilateral triangles of same size.
\end{itemize}
\end{corollary}

\begin{proof}
We know from Lemma~\ref{lem:Aligned} that every critical trajectory of $\widetilde{\mathcal{W}}^{(0)}_\vartheta$ can actually be thought about as a trajectory starting at $A$. Moreover, since these trajectories are straight lines on a torus with rational slope, eventually return to their initial point; since they are ramified, they pass $B$ and $C$ to form a chain of three saddle connections. Thus $\widetilde{\mathcal{W}}^{(0)}_\vartheta$ is formed by exactly three walls, all of which are double.
\par
It follows from Lemma~\ref{lem:TripleIntersection} that these three double walls intersect pairwise exactly at the same $\Delta_\vartheta$ points (including $A,B,C$). Since there are already three double walls passing through these points, no additional trajectory can appear, and $\widetilde{\mathcal{W}}_\vartheta$ coincides with $\widetilde{\mathcal{W}}^{(0)}_\vartheta$.
\par
The corners of each connected component of $\Sigma \setminus \widetilde{\mathcal{W}_\vartheta}$ are of angle $\frac{\pi}{3}$. It follows that each of them is an equilateral triangle. Each of the $\Delta_\vartheta$ intersection points is a vertex with 6 half-edges, giving $3\Delta_\vartheta$ edges, and since $\chi(\Sigma)=0$, we have by Euler's formula that the number of triangles in the resulting triangulation is
$2\Delta_\vartheta$.
\end{proof}

We can construct the spectral network by hand following straight line trajectories on the torus; we give an example in Figure \ref{fig:ramifiednetworks_byhand} with $\Delta_\vartheta=3$ and $2\Delta_\vartheta=6$ triangles as expected.
    \begin{figure}[h]
        \centering
        \begin{subfigure}{0.35\textwidth}
                \includegraphics[width=\linewidth]{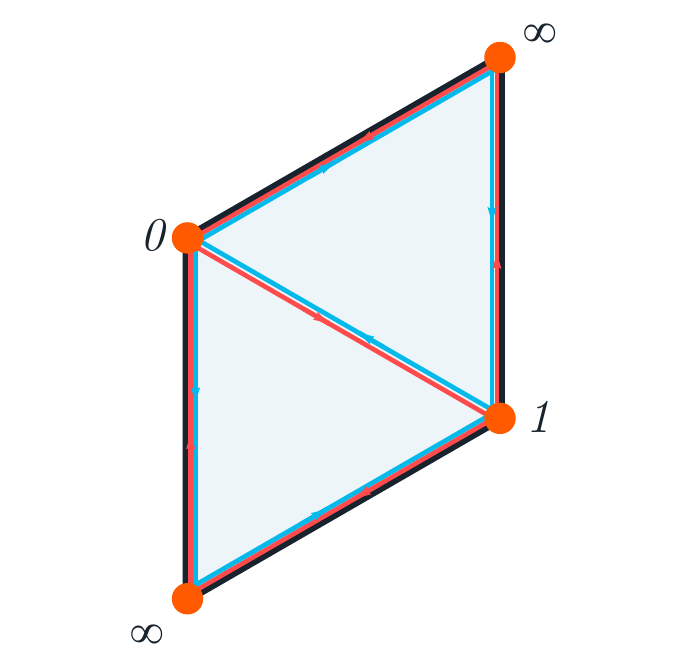}
                \caption{$(p,q)=(1,2)$ on $X=\mathbf{P}^1$}
                \label{fig:ramifiedbasep1network}
        \end{subfigure}
        \hspace{1.5cm}
        \begin{subfigure}{0.50\textwidth}
        \includegraphics[width=\linewidth]{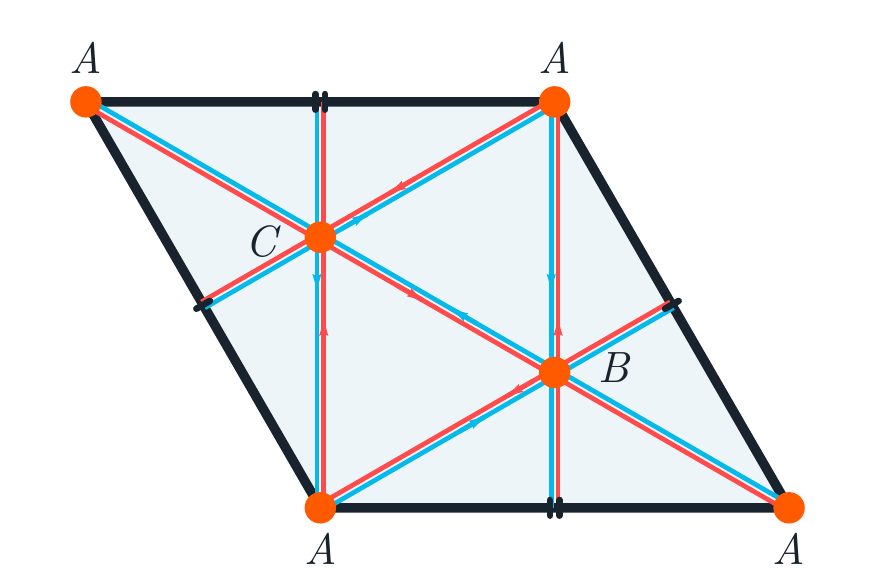}            \caption{$\widetilde{\mathcal{W}}_\vartheta$ on $\Sigma$.}\label{fig:ramifiedcovernetwork}
        \end{subfigure}

        \caption{Networks on $X$ and $\Sigma$ in the ramified case.}
        \label{fig:ramifiednetworks_byhand}
    \end{figure}

\subsubsection*{Unramified rational slopes}

Following Lemma~\ref{lem:Aligned}, when $\vartheta$ is unramified, trajectories of $\widetilde{\mathcal{W}}^{(0)}_\vartheta$ are saddle connections of $\Sigma$, each containing exactly one point among $A$, $B$ and $C$. Thus, $\widetilde{\mathcal{W}}^{(0)}_\vartheta$ is formed by exactly $9$ double walls. We denote by $\ell_{P,i}$ the double wall (in $\widetilde{\mathcal{W}}^{(0)}_\vartheta$) containing $P \in \lbrace{A,B,C}\rbrace$ of slope $R^{i-1}\vartheta$ for $i \in \lbrace 1,2,3 \rbrace$.
\par
We deduce from Lemma~\ref{lem:TripleIntersection} that for each $P \in \{A,B,C\}$, the three double walls $\ell_{P,1}, \ell_{P,2}, \ell_{P,3}$ have only triple intersection points. There are $\Delta_\vartheta$ such points, denoted by $I_{P,1},\dots,I_{P,\Delta_{\vartheta}}$ (the point $P$ is one of them). From the previous discussion, we deduce that $I_{P,i}$ and $I_{Q,j}$ are always distinct whenever $P \neq Q$. These $3\Delta_{\vartheta}$ points are intersections in $\widetilde{\mathcal{W}}^{(0)}_\vartheta$, but there are additional ones.

\begin{lemma}\label{lem:SecondaryTripleIntersection}
Any $O \in \ell_{A,1} \cap \ell_{B,2}$ is contained in $\ell_{C,3}$.
\end{lemma}

\begin{proof}A direct computation shows that $O$ is of the form $\left(\frac{i p +(j-i)q}{\Delta_{\vartheta}}+\frac{2q-p}{3\Delta_{\vartheta}}\right)(p,q)$, for $i,j\in\mathbb{Z}$.
Now, we consider the leaf $\ell$ of slope $R^{2}(\vartheta)$ containing $O$. Choosing $(q-p,-p)$ as a direction vector of $\ell$, $C=(\frac{1}{3},\frac{2}{3})$ belongs to $\ell$ if and only if there exists $t \in \mathbb{R}$ and $m,n \in \mathbb{Z}$ such that $$\frac{3(i p +(j-i)q)+2q-p}{3\Delta_{\vartheta}}(p,q)+t(q-p,-p) = (\frac{1}{3},\frac{2}{3}) + (m,n).$$
{Supposing $p-q\neq0$, direct computation shows that this equation has a solution if and only if $n=j-i+\frac{p(j-m)}{q-p}$, which can always be ensured to be integer by choosing $m$ appropriately.
On the other hand, if $p=q$, the equation simplifies to 
\[
\left(j+\frac{1}{3}\right)(1,1)+(0,-t p) = (\frac{1}{3},\frac{2}{3})+(m,n)
\]}
\noindent so that choosing $m=j$ to satisfy the first equation, we are left with $j+\frac{1}{3}-t p = \frac{2}{3}+n$, which always has a solution.

\end{proof}

Analogous to Corollary~\ref{cor:RamifiedSlope}, we conclude that 

 \begin{corollary}\label{cor:UNRamifiedSlope}
Consider an unramified rational slope $\vartheta$ with Eisenstein determinant $\Delta_{\vartheta}$. Then, the lifted spectral network $\widetilde{\mathcal{W}}_\vartheta$ on $\Sigma$ satisfies the following conditions:
\begin{itemize}
    \item Every trajectory of $\widetilde{\mathcal{W}_\vartheta}$ is a double wall;
    \item $\widetilde{\mathcal{W}_\vartheta}$ coincides with $\widetilde{\mathcal{W}}^{(0)}_\vartheta$;
    \item $\widetilde{\mathcal{W}_\vartheta}$     decomposes $\Sigma$ into $18\Delta_{\vartheta}$ equilateral triangles of same size.
\end{itemize}
\end{corollary}
\begin{proof}

Consider the leaves $\ell_{P,i}, \ell_{Q,j}$ for any $P,Q \in \{A,B,C\}$ and $i,j\in\{1,2,3\}$. If $i=j$ then the leaves are parallel, so either coincide or never intersect (except possibly at $P,Q$). If $i\neq j$, and $P=Q$, then Lemma \ref{lem:TripleIntersection} tells us there are $\Delta_\vartheta$ intersection points. In fact, a variant of that argument tells us that for any $P \neq Q$ and $i \neq j$, the leaf $\ell_{P,i}$ intersects with $\ell_{Q,j}$ in exactly $\Delta_{\vartheta}$ points. Applying Lemma~\ref{lem:SecondaryTripleIntersection}, we obtain that these intersection points are actually triple intersections. Thus, no new walls are formed at joints, and all walls are saddle connections in $\widetilde{\mathcal{W}}_\vartheta^{(0)}$.

Furthermore, all these triple points are clearly distinct from the points of the form $I_{P,k}$, $k=1,\ldots \Delta_\vartheta$ discussed above. Each triple point accounts for three pairwise intersections so there are $6\Delta_{\vartheta}$ such additional triple points.
\par
In total, the $9$ double walls of $\widetilde{\mathcal{W}}^{(0)}_\vartheta$ intersect in $9\Delta_{\vartheta}$ vertices and all of them are triple points. It follows that no additional trajectory appears there and $\widetilde{\mathcal{W}}_\vartheta$ coincides with $\widetilde{\mathcal{W}}^{(0)}_\vartheta$.

The final count follows again from an Euler characteristic calculation. Each of the $9\Delta_{\vartheta}$ vertices is a triple point so it is incident to $6$ half-edges. It follows that $\widetilde{\mathcal{W}}^{(0)}_\vartheta$ is formed by $27\Delta_{\vartheta}$ edges. Then, the number of faces in the resulting triangulation is $27\Delta_{\vartheta}-9\Delta_\vartheta =18\Delta_\vartheta$.
\end{proof}

    \begin{figure}[h]
        \centering
        \begin{subfigure}{0.35\textwidth}
                \includegraphics[width=\linewidth]{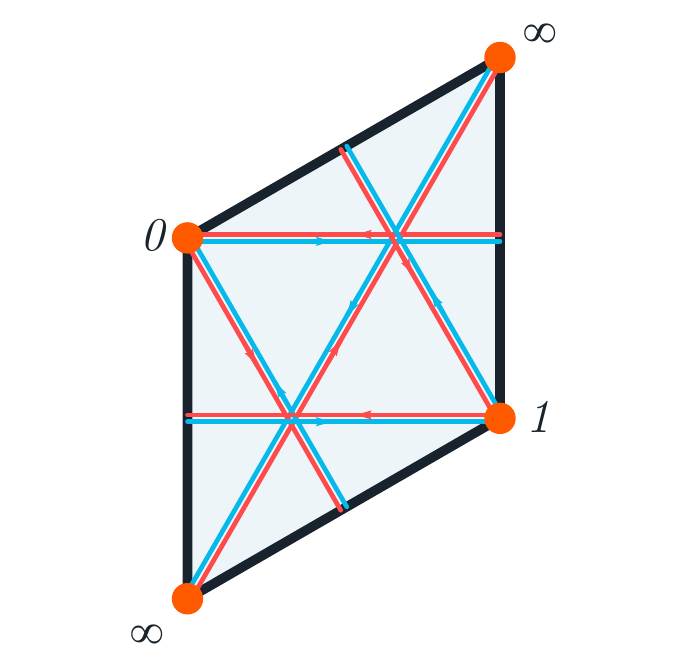}
                \caption{$(p,q)=(1,1)$ on $X=\mathbf{P}^1$}
                \label{fig:basep1network}
        \end{subfigure}
        \hspace{1.5cm}
        \begin{subfigure}{0.50\textwidth}
        \includegraphics[width=\linewidth]{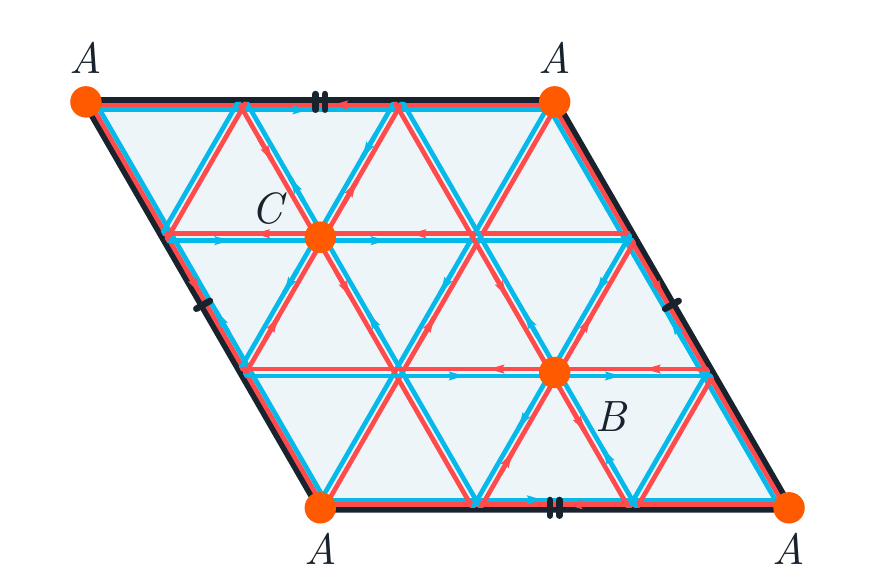}            \caption{$\widetilde{\mathcal{W}}_\vartheta$ on $\Sigma$.}\label{fig:covernetwork}
        \end{subfigure}

        \caption{Networks on $X$ and $\Sigma$ in the unramified case.}
        \label{fig:networks_byhand}
    \end{figure}

\subsection{Lifted spectral networks: the case of irrational slopes}\label{sub:irrational}

In this section, we describe the lifted spectral network $\widetilde{\mathcal{W}}_\vartheta$ for an irrational slope  $\vartheta$. 

\begin{proposition}\label{prop:IrrationalSlope}
Consider an irrational slope $\vartheta$. Then, the lifted spectral network $\widetilde{\mathcal{W}_\vartheta}$ on $\Sigma$ satisfies the following conditions:
\begin{itemize}
    \item for any $k \in \mathbb{N}$, $\widetilde{\mathcal{W}}^{(k)}_\vartheta$ is dense in $\Sigma$;
    \item $\widetilde{\mathcal{W}_\vartheta}$ does not contain any double wall.
\end{itemize}
\end{proposition}

\begin{proof}
Following Lemma~\ref{lem:Dichotomy}, any trajectory of the lifted spectral network $\widetilde{\mathcal{W}_\vartheta}$ is dense in $\Sigma$.
\par
If there is a double wall in $\widetilde{\mathcal{W}}_\vartheta$, its projection must be a double wall in $\mathcal{W}_\vartheta$, so by Lemma \ref{lem:walltoweb} we obtain a finite web supported on $\mathcal{W}_\vartheta$. To every finite web we have associated its canonical lifts, which are nonzero integral homology classes of phase $\vartheta+\frac{\pi}{2}$ (up to rotation by $\frac{2\pi}{6})$ in $H_1(\Sigma,\mathbb{Z})$ by Lemma \ref{lem:nonzero}. But such a number must be of rational phase, since every homology class can be represented by a straight segment between the lattice points of $\mathbb Z+\omega \mathbb{Z}$, contradicting the hypothesis.
\end{proof}

\subsection{Proof of the main theorem}

In Sections~\ref{sub:rational} and~\ref{sub:irrational}, we described the lifted spectral network $\widetilde{\mathcal{W}_\vartheta}$ on the canonical cover $\Sigma$ in both the rational and irrational cases. We now project $\widetilde{\mathcal{W}_\vartheta}$ onto $X$ in order to obtain the spectral network $\mathcal{W}_\vartheta$ and its properties.

\begin{theorem}\label{thm:MAINbody}
The spectral network $\mathcal{W}_\vartheta(\varphi)$ satisfies the following properties.
\par
If $\vartheta$ is rational with Eisenstein determinant $\Delta_{\vartheta}$, then $\mathcal{W}_\vartheta(\varphi)$ coincides with $\mathcal{W}^{(0)}_\vartheta(\varphi)$ and every trajectory of $\mathcal{W}_\vartheta(\varphi)$ is a double wall in the image of some finite web. Moreover, $\mathcal{W}_\vartheta(\varphi)$ decomposes $X$ into
\begin{itemize}
    \item $\frac{2}{3}\Delta_{\vartheta}$ equilateral triangles of the same size if $\vartheta$ is ramified (i.e. $3 \mid \Delta_{\vartheta}$);
    \item $6\Delta_{\vartheta}$ equilateral triangles of the same size if $\vartheta$ is unramified (i.e. $3 \nmid \Delta_{\vartheta}$).
\end{itemize}

If $\vartheta$ is irrational, then $\mathcal{W}_\vartheta(\varphi)$ is dense in $X$ and does not contain any double wall (nor finite web).
\end{theorem}

\begin{proof}
Lemma~\ref{lem:LiftProjection} proves that $\mathcal{W}_\vartheta$ coincides with the projection of $\widetilde{\mathcal{W}_\vartheta}$ to $X$.
\par
If $\vartheta$ is rational of Eisenstein determinant $\Delta_{\vartheta}$, Corollaries~\ref{cor:RamifiedSlope} and~\ref{cor:UNRamifiedSlope} describe $\widetilde{\mathcal{W}_\vartheta}$:
\begin{itemize}
    \item $\widetilde{\mathcal{W}}_\vartheta$ coincides with $\widetilde{\mathcal{W}}^{(0)}_\vartheta$;
    \item $\widetilde{\mathcal{W}_\vartheta}$ is formed by double walls only;
    \item $\widetilde{\mathcal{W}}_\vartheta$ decomposes $\Sigma$ into $2\Delta_\vartheta$ or $18\Delta_\vartheta$ equilateral triangles of same size (depending on whether $\vartheta$ is ramified or unramified).
\end{itemize}
Thus, the same properties hold for its projection
as formed by double walls $\mathcal{W}_\vartheta$. As the ramification points are vertices of the decomposition, we just divide the number of triangles by $3$. By Lemma \ref{lem:walltoweb}, the double walls must be in the image of some finite web.
\par
If $\vartheta$ is irrational, the claim follows from Proposition~\ref{prop:IrrationalSlope}.
\end{proof}

\appendix
\section{Comparison with Hollands-Neitzke's parametrization}\label{sec:conventions}

Since our spectral networks are rotated compared to those of \cite{Hollands:2016kgm}, the $p,q$ appearing in this paper are not the same as the ones there. Our $\vartheta$ is related to theirs by  $\vartheta^{\mathrm{HN}}=\vartheta+\frac{\pi}{2}$, so that

\[
\tan\vartheta
=-\frac{1}{\tan{\vartheta^{\mathrm{HN}}}}
\]

\noindent Let us use $\mathsf{p}$, $\mathsf{q}$ for their parametrization \cite{Hollands:2016kgm}, i.e.
\[
\tan \vartheta^{\mathrm{HN}}={\sqrt{3}}\frac{\mathsf{q}}{\mathsf{q}-2\mathsf{p}}
\]
with $\mathsf{p},\mathsf{q}$ coprime, and continue to use $p$ and $q$ following our convention as in the body of the paper. Equating both sides we rearrange to obtain 

\[\frac{{p}}{{q}}=\frac{\mathsf{p}+\mathsf{q}}{2\mathsf{p}-\mathsf{q}}\]

It is easy to see that $3\mid \mathsf{p}+\mathsf{q}$ if and only if $3 \mid 2\mathsf{p}-\mathsf{q}$. Since $\mathsf{p},\mathsf{q}$ are coprime, there are two cases: either $\mathsf{p}+\mathsf{q}$ and $2\mathsf{p}-\mathsf{q}$ are coprime (thus both not divisible by $3$), or their greatest\footnote{Suppose $\mathsf{p}+\mathsf{q} = d\cdot k$, $2\mathsf{p}-\mathsf{q}= d \cdot l$ for some positive integer $d$. Then adding and subtracting the equations we get $3\mathsf{p}$ and $3\mathsf{q}$ are both multiples of $d$, violating coprimality unless $d=1,3$.} common divisor is $3$. Thus we have (up to signs)

\[(p,q)=\begin{cases}   \,\,\,\,(\mathsf{p}+\mathsf{q},2\mathsf{p}-\mathsf{q}), \quad &3\nmid \mathsf{p}+\mathsf{q} \\
   \frac{1}{3}(\mathsf{p}+\mathsf{q},2\mathsf{p}-\mathsf{q}), \quad &3 \,\mid\; \mathsf{p}+\mathsf{q},
\end{cases}\]
In fact, $\Delta_{\vartheta}^{\mathrm{HN}}:=\mathsf{p}^2-\mathsf{p}\mathsf{q}+\mathsf{q}^2$ is divisible by $3$ if and only if $\Delta_{\vartheta}$ is not divisible by $3$; the same quadratic expression appears in both conventions, but with the opposite meaning.

With this dictionary, we can check exactly that every network appearing in their paper is of the type required by our theorem:

\[
\begin{array}{c|c|c|c|c|c}
(\mathsf{p},\mathsf{q}) & (p,q)
& p^2-pq+q^2
& \Delta \bmod 3 & \text{\# transverse double walls} & \text{\# triangles}\\  \hline
(1,0) & (1,2)  & 3  & 0 & 1 & 2\\
(1,1) & (2,1) & 3  & 0 & 1 & 2\\
(1,2) & (1,0)  & 1  & 1 & 3 & 6\\
(1,3) & (4,-1)  & 21 & 0 & 1 & 14 \\
(2,5) & (7,-1)  & 57 & 0 & 1 & 38
\end{array}
\]

For comparison to concrete numerical calculations of spectral networks, we include some figures plotted using \cite{neitzkeswn}. The plots below are taken directly from \cite{Hollands:2016kgm}, reproduced with permission (we have removed some labels that we refrained from introducing). For numerical convenience, the conical singularities are placed at $1,\omega,\omega^2$ by an automorphism of $\mathbf{P}^1$.

\begin{figure}[h]
    \centering

            \begin{subfigure}{0.4\textwidth}
                \includegraphics[width=0.7\linewidth]{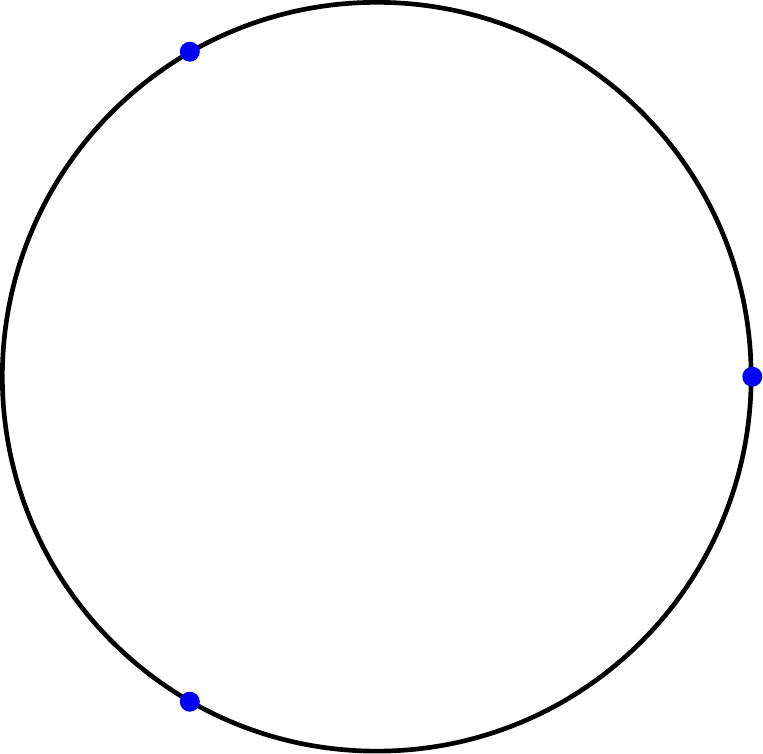}
                \caption{$(\mathsf{p},\mathsf{q})=(1,0)$, i.e. $(p,q)=(1,2)$.}
                \label{fig:circlenetwork}
        \end{subfigure}
        \hspace{1.5cm}
        \begin{subfigure}{0.4\textwidth}
        \includegraphics[width=0.85\linewidth]{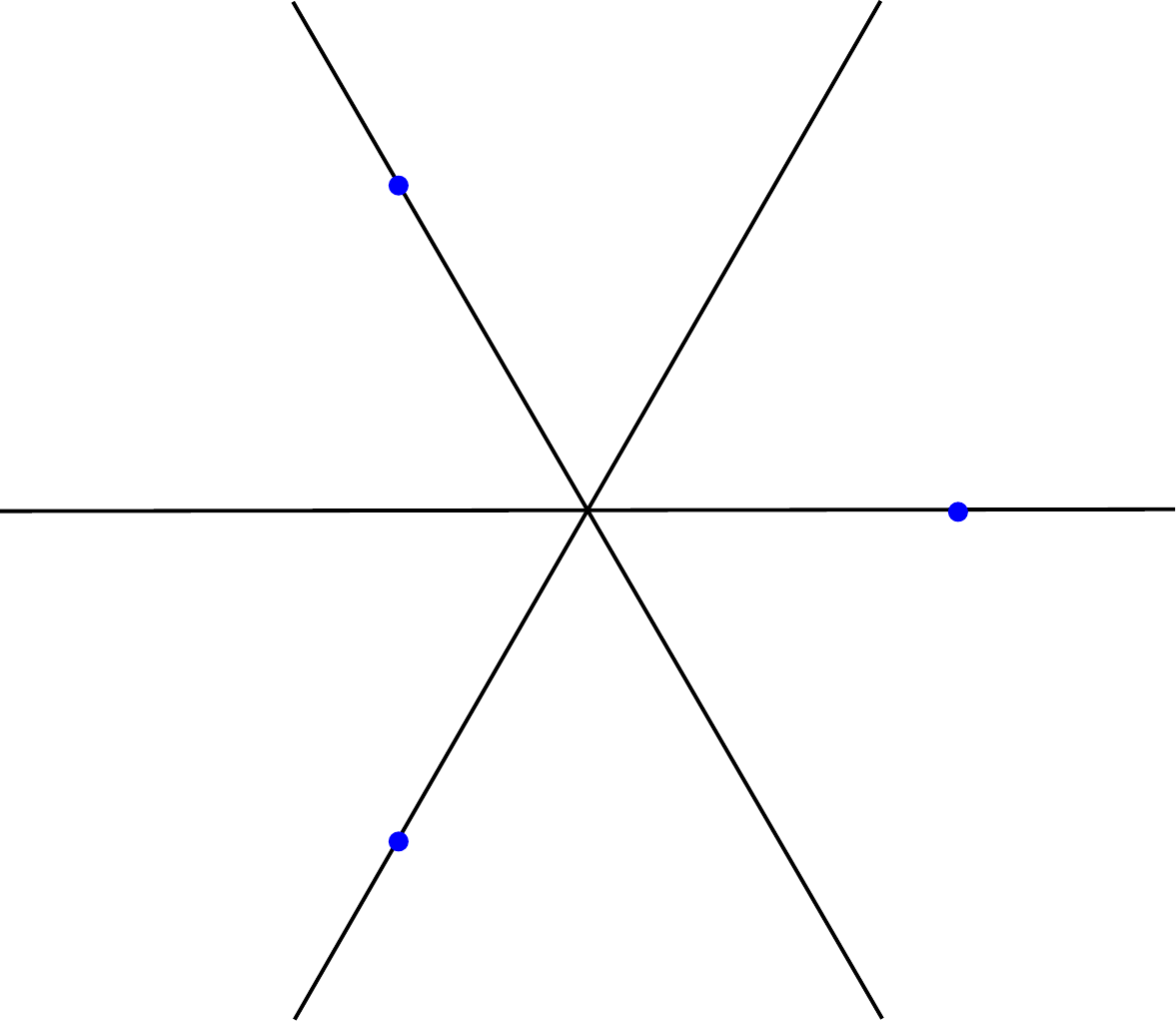}            \caption{$(\mathsf{p},\mathsf{q})=(1,2)$, i.e. $(p,q)=(1,0)$.}\label{fig:starnetwork}
        \end{subfigure}
\vspace{0.5cm}
        
            \begin{subfigure}{0.45\textwidth}
                \includegraphics[width=0.9\linewidth]{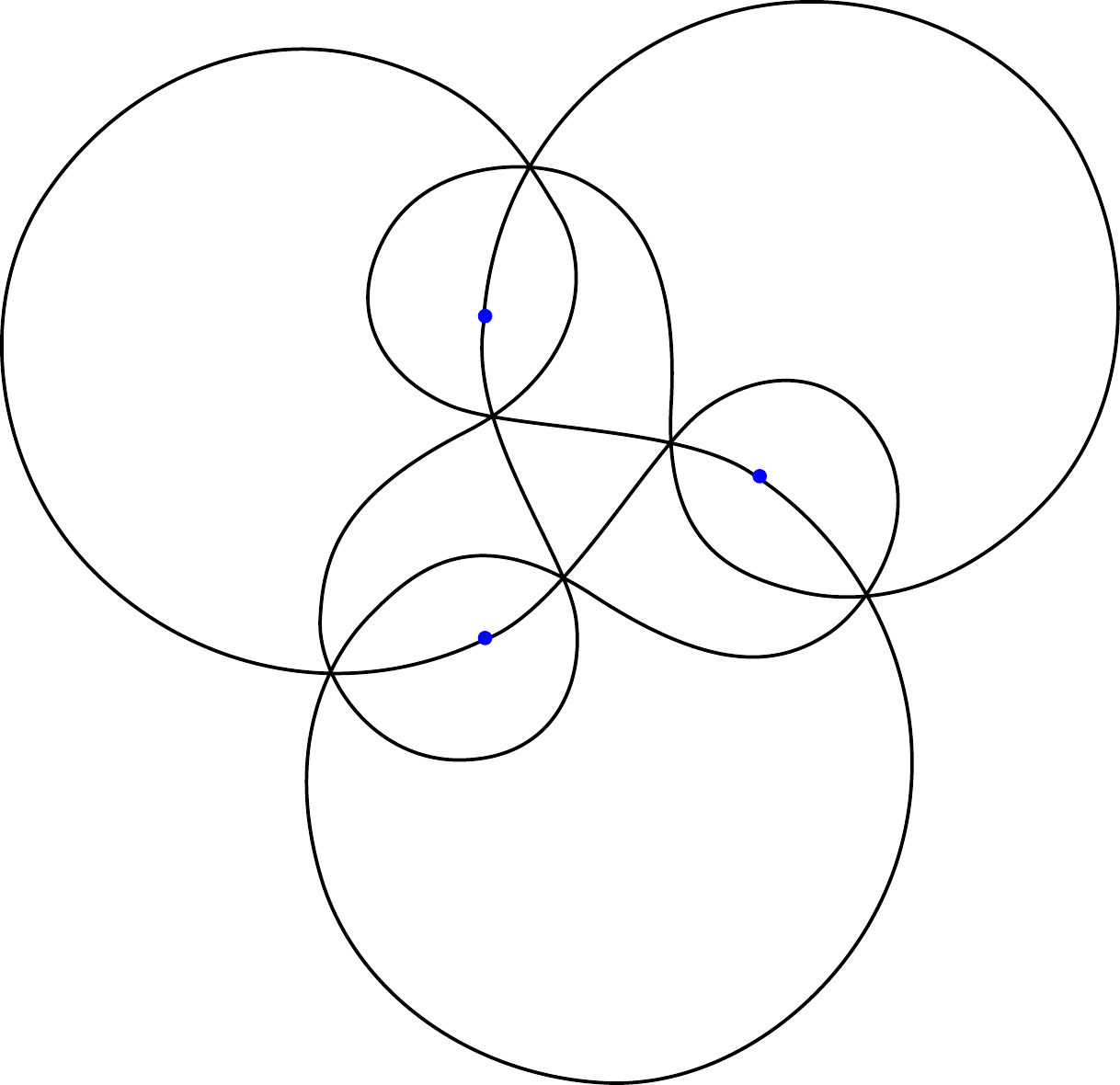}
                \caption{$(\mathsf{p},\mathsf{q})=(1,3)$, i.e. $(p,q)=(4,-1)$.}
                \label{fig:232network}
        \end{subfigure}
        \hspace{0.5cm}
        \begin{subfigure}{0.45\textwidth}
        \includegraphics[width=0.9\linewidth]{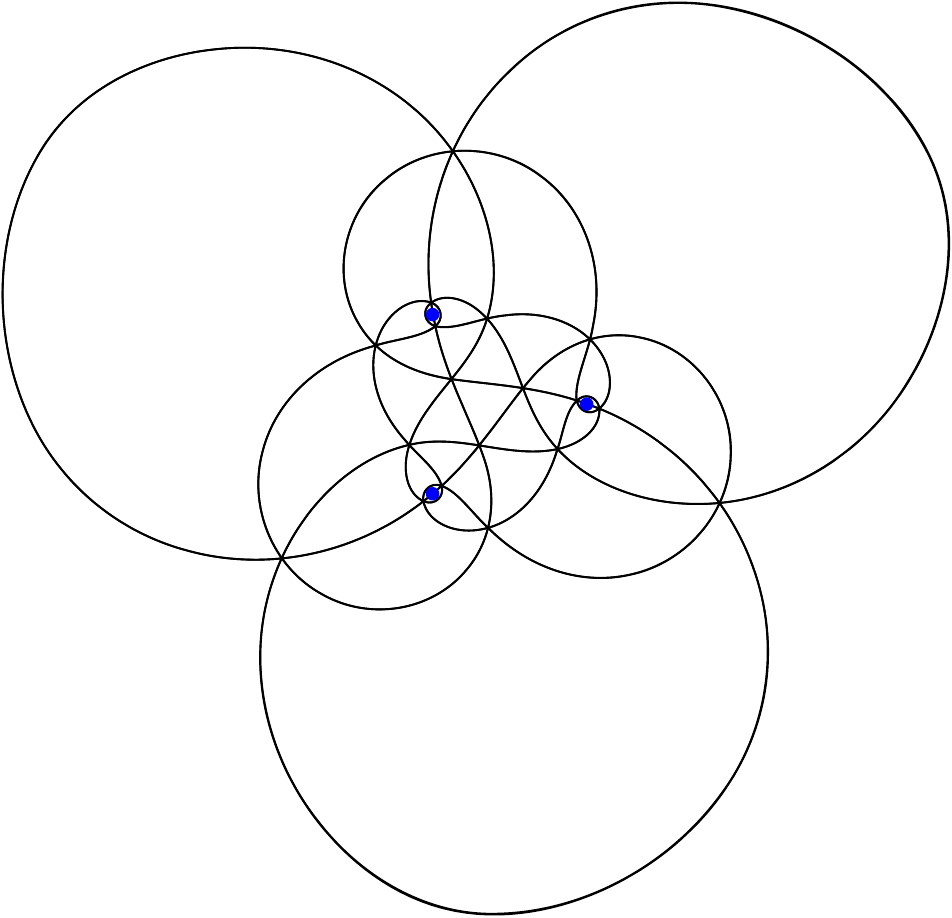}            \caption{$(\mathsf{p},\mathsf{q})=(2,5)$, i.e. $(p,q)=(7,-1)$.}\label{fig:25network}
        \end{subfigure}

        \caption{Networks plotted numerically by Hollands-Neitzke.}
\end{figure}

The reader is invited to count the triangles.

\bibliography{refs}
\bibliographystyle{utphys.bst}

\end{document}